\documentclass[11pt]{amsart}
\usepackage{amssymb,latexsym,graphicx}
\usepackage{enumerate,xypic,hyperref}

\makeatletter
\@namedef{subjclassname@2010}{%
  \textup{2010} Mathematics Subject Classification}
\makeatother

\newtheorem{thm}[equation]{Theorem}
\newtheorem{cor}[equation]{Corollary}
\newtheorem{prop}[equation]{Proposition}
\newtheorem{lem}[equation]{Lemma}

\newtheorem{question}[equation]{Question}
\newtheorem{conjecture}[equation]{Conjecture}

\numberwithin{equation}{section}

\theoremstyle{definition}
\newtheorem{defin}[equation]{Definition}
\newtheorem{rem}[equation]{Remark}

\newcommand{\A}{\ensuremath{{\mathbb{A}}}}
\newcommand{\C}{\ensuremath{{\mathbb{C}}}}
\newcommand{\Z}{\ensuremath{{\mathbb{Z}}}}
\renewcommand{\P}{\ensuremath{{\mathbb{P}}}}
\newcommand{\Q}{\ensuremath{{\mathbb{Q}}}}

\newcommand{\F}{\ensuremath{{\mathbb{F}}}}
\newcommand{\G}{\ensuremath{{\mathbb{G}}}}

\DeclareMathOperator{\End}{End}

\DeclareMathOperator{\Per}{Per}

\DeclareMathOperator{\Spec}{Spec}

\DeclareMathOperator{\SL}{SL}

\DeclareMathOperator{\Aut}{Aut}

\DeclareMathOperator{\tr}{tr}
\DeclareMathOperator{\ord}{ord}

\begin{document}


\title[Zeta Functions of Dynamically Affine Maps]{The Artin-Mazur Zeta Function of a Dynamically Affine Rational Map in Positive Characteristic}

\author[A. Bridy]{Andrew Bridy}
\address{Andrew Bridy\\Department of Mathematics\\ University of Wisconsin-Madison\\
Madison, WI 53706, USA}
\email{bridy@math.wisc.edu}

\date{}

\begin{abstract}
A dynamically affine map is a finite quotient of an affine morphism of an algebraic group. We determine the rationality or transcendence of the Artin-Mazur zeta function of a dynamically affine self-map of $\P^1(k)$ for $k$ an algebraically closed field of positive characteristic.
\end{abstract}

\subjclass[2010]{Primary 37P05; Secondary 11G20, 11B85}

\keywords{Arithmetic Dynamics, Algebraic Groups, Automatic Sequences, Finite Fields}

\maketitle

\section{The Artin-Mazur Zeta Function}

Let $X$ be a set and let $f:X\to X$ define a dynamical system. Let $\Per_n(f)=\{x\in X:f^n(x)=x\}$, where $f^n$ denotes the composition of $f$ with itself $n$ times. The Artin-Mazur zeta function of this dynamical system is the formal power series given by
\begin{equation*}
\zeta(f,X;t)=\exp\left( \sum_{n=1}^\infty \#\Per_n(f)\frac{t^n}{n} \right).
\end{equation*}
Assume that $\#\Per_n(f)<\infty$ for all $n$, as otherwise $\zeta$ is not defined. The power series $\zeta(f,X;t)$ has rational coefficients, and it is not hard to show that $\zeta(f,X;t)\in\Z[[t]]$ by means of the product formula
\begin{equation*}
\zeta(f,X;t)=\prod_{\text{cycles }C}\left(1-t^{|C|}\right)^{-1}.
\end{equation*}

This zeta function was introduced by Artin and Mazur in ~\cite{ArtinMazur}, where it is studied for $X$ a manifold and $f$ a diffeomorphism. In this setting, only the \emph{isolated} periodic points are counted. This will not be an important distinction for our purposes.

This paper continues the study of the following question, introduced in ~\cite{BridyActa} for polynomials, but just as easily phrased for rational functions.

\begin{question}
For which $f\in{\overline{\F}}_p(x)$ is $\zeta(f,\P^1(\overline{\F}_p);t)$ rational?
\end{question}

The purpose of this paper is to answer this question for rational maps that are dynamically affine. These are maps that, loosely speaking, come from endomorphisms of algebraic groups; a precise definition will be given in section ~\ref{sec: affine}. There are five families of dynamically affine maps in one dimension: power maps, Chebyshev polynomials, Latt\`es maps, additive polynomials, and subadditive polynomials. (To be precise, this classification only holds up to conjugacy by $\Aut(\P^1)$, but $\zeta$ is a conjugacy invariant.) We determine the rationality of the zeta function for each of these families, and we show that when it fails to be rational, it is transcendental.

Let $k$ be an arbitrary algebraically closed field of characteristic $p$. Our main results are the following.

\begin{thm}\label{thm: main}
Let $f\in k(x)$ be a separable power map, Chebyshev polynomial, or Latt\`es map. Then $\zeta(f,\P^1(k);t)$ is transcendental over $\Q(t)$.
\end{thm}

\begin{thm}\label{thm: main2}
Let $f\in k[x]$ be a separable additive or subadditive polynomial. If $f'(0)$ is algebraic over $\F_p$, then $\zeta(f,\P^1(k);t)$ is transcendental over $\Q(t)$. If $f'(0)$ is transcendental over $\F_p$, then $\zeta(f,\P^1(k);t)$ is rational.
\end{thm}

These theorems can be seen as the broadest possible generalization of the work in ~\cite{BridyActa}, as the maps considered there are very specific cases of one-dimensional dynamically affine maps. In order to handle the new cases, we need to study the arithmetic of the endomorphism rings of one-dimensional algebraic groups, which can be somewhat complicated in the case of an elliptic curve.

Inseparable maps are excluded from the above theorems because their zeta functions are trivially rational. Let $f\in k(x)$ be inseparable and of degree $d$. The derivative of $f$ is identically zero, so $f^n(x)-x$ has distinct roots for every $n$ and $\#\Per_n(f)=d^n+1$. Therefore
\begin{equation*}
\zeta(f,\P^1(k);t)=\exp\left( \sum_{n=1}^\infty\frac{(d^n+1)t^n}{n}\right) = \frac{1}{(1-t)(1-dt)}.
\end{equation*}
For a general $f\in k(x)$, it is not always the case that $\#\Per_n(f)=d^n+1$, but it is certainly true that $\#\Per_n(f)\leq d^n+1$. Therefore if we consider $\zeta(f,\P^1(k);t)$ as a function of a complex variable, it converges to a holomorphic function in a positive radius around the origin.

\begin{rem}
In higher dimensions, the formula $\deg(f^n)=(\deg f)^n$ is not necessarily true. This complicates the above calculation of rationality. See, for example, ~\cite{Bellon1999}, ~\cite{Hasselblatt2007}, and ~\cite{SilvermanEntropy} for a discussion of this phenomenon.
\end{rem}

\begin{rem}
In characteristic zero, the situation is very different. Hinkkanen shows that every $f\in\C(x)$ has a rational zeta function ~\cite{Hinkkanen}. The proof relies on the fact that there are only finitely many $x\in\P^1(\C)$ such that $(f^n)'(x)$ is a root of unity. This argument fails catastrophically in positive characteristic because every element of $\overline{\F}_p$ is a root of unity. Nevertheless, it is peculiar that the conclusion of Hinkkanen's theorem holds in our setting almost exclusively when $f$ is inseparable, which is a phenomenon that cannot occur in characteristic 0.
\end{rem}

The rest of the paper will prove Theorems ~\ref{thm: main} and ~\ref{thm: main2}. In Section ~\ref{sec: affine} we define dynamically affine maps, and in Sections ~\ref{sec: G_m}, ~\ref{sec: G_a}, and ~\ref{sec: Lattes} we classify all dynamically affine maps of $\P^1$ and establish some facts about their periodic points. A crucial and interesting feature of these maps is that we can count their periodic points by studying the arithmetic of certain endomorphism rings. Section ~\ref{sec: lifting the exponent} provides some algebraic lemmas that are useful in this direction. Our proof also employs the theory of sequences generated by finite automata. Section \ref{sec: Automatic Sequences} sketches the necessary background in this area. Sections \ref{sec: Proof 1.2} and \ref{sec: Proof 1.3} finish the proof of Theorems ~\ref{thm: main} and ~\ref{thm: main2}.

The results of this paper suggest the following conjecture.
\begin{conjecture}\label{conj: separable}
If $f\in \overline{\F}_p(x)$ is separable, $\zeta(f,\P^1(\overline{\F}_p);t)$ is transcendental over $\Q(t)$.
\end{conjecture}
If $f$ is not dynamically affine, the size of $\Per_n(f)$ can vary wildly as $n$ increases, making it difficult to determine the algebraic structure of the zeta function. Even the low degree map $f(x)=x^2+1$ behaves very irregularly in its periodic point counts (for $p\notin\{2,3\})$, so the nature of $\zeta$ is unclear. However, note that by Theorem ~\ref{thm: main2} the above conjecture is false if we replace $\overline{\F}_p$ by an algebraically closed field $k$ that is transcendental over $\F_p$.

\section{Overview of Dynamically Affine Maps}\label{sec: affine}

The following definitions are taken from ~\cite[Ch 6.8]{SilvermanADS}.
\begin{defin}
Let $G$ be a commutative algebraic group. An \emph{affine morphism} of $G$ is a map $\psi:G\to G$ that can be written as a composition of a finite endomorphism of degree at least 2 and a translation.
\end{defin}
\begin{defin}
Let $V$ be a variety. A morphism $f:V\to V$ is \emph{dynamically affine} if there exist a connected commutative algebraic group $G$, an affine morphism $\psi:G\to G$, a finite subgroup $\Gamma\subseteq\Aut(G)$ and a morphism $\pi:G\to G/\Gamma$, and a morphism that identifies $G/\Gamma$ with a Zariski dense open subset of $V$ such that the following diagram commutes:
\begin{equation}\label{eqn: dynamically affine}
\xymatrix{
 G \ar[r]^\psi\ar[d]^\pi & G \ar[d]^\pi \\
G/\Gamma \ar[r]\ar[d] & G/\Gamma \ar[d] \\
V \ar[r]^f & V
}
\end{equation}
\end{defin}

It is well known that the only dynamically affine maps of $\P^1(\C)$ are power maps, Chebyshev polynomials, and Latt\`es maps (up to conjugacy by fractional linear transformations) ~\cite[p. 378]{SilvermanADS}. These arise when $G$ is either the multiplicative group $\G_m$ or an elliptic curve.

In characteristic $p$, there are two additional families of dynamically affine maps, both of which arise from the additive group $\G_a$. These are additive polynomials, which are maps such as $f(x)=x^p-x$ that distribute over addition, and subadditive polynomials such as $f(x)=x(x-1)^{p-1}$, which arise as the maps induced by additive polynomials on the quotient of $\G_a$ by a group of roots of unity.

We elaborate on these families in the sections that follow. The only one-dimensional algebraic groups are $\G_m$, $\G_a$, and elliptic curves. By considering all of the possibilities for the group $\Gamma$, we show that the maps listed above are all of the dynamically affine maps of $\P^1$. First, however, we establish a lemma that counts $\Per_n(f)$ in terms of the kernels of endomorphisms of $G$.

\begin{lem}\label{lem: periodic count}
Let $f:V\to V$ be dynamically affine. Assume that the affine morphism $\psi:G\to G$ is surjective. Write $\psi$ as $\psi(g)=\sigma(g)+h$, where $\sigma\in\End(G)$, $h\in G$ and the group law of $G$ is written additively. Then
\begin{equation*}
\#\Per_n(f)=\#(\Per_n(f) \setminus (G/\Gamma)) + \frac{1}{|\Gamma|}\sum_{\gamma\in\Gamma}\#\ker(\sigma^n-\gamma).
\end{equation*}
\end{lem}

\begin{proof}
Recall that $G/\Gamma$ is identified with a Zariski open subset of $V$. The equation above claims that the $n$-periodic points that lie in this set are counted by the formula $\frac{1}{|\Gamma|}\sum_{\gamma\in\Gamma}\#\ker(\sigma^n-\gamma)$.

Suppose that $z\in\Per_n(f)\cap G/\Gamma$. By diagram ~\ref{eqn: dynamically affine}, there exists $g\in \pi^{-1}(z)$ and $\gamma\in\Gamma$ such that $\psi^n(g)=\gamma(g)$, and every such choice of $g$ and $\gamma$ gives some $z\in\Per_n(f)\cap G/\Gamma$. Therefore
\begin{equation*}\label{eqn: Fix Count}
\Per_n(f)\cap G/\Gamma =\pi(\{g\in G:\psi^n(g)=\gamma(g)\text{ for some }\gamma\in\Gamma\}).
\end{equation*}

By slight abuse of notation, let $\ker(\psi^n-\gamma)=(\psi^n-\gamma)^{-1}(0)$. Define 
\begin{equation*}
S=\bigcup_{\gamma\in\Gamma}\ker(\psi^n-\gamma),
\end{equation*}
so that $\Per_n(f)\cap G/\Gamma=\pi(S)$. We claim that $\Gamma$ acts on $S$. 

Let $g\in S$, so that $\psi^n(g)=\delta(g)$ for some $\delta\in\Gamma$. Let $\gamma\in\Gamma$. Observe that
\begin{equation*}
\pi(\psi^n(\gamma(g))) = f^n(\pi(\gamma(g)))=f^n(\pi(g))=\pi(\psi^n(g)).
\end{equation*}
As $\psi^n(\gamma(g))$ and $\psi^n(g)$ have the same image under $\pi$, there exists some $\delta'\in\Gamma$ such that $\psi^n(\gamma(g))=\delta'(g)$, and therefore $\gamma(g)\in S$. 
(Somewhat surprisingly, $\delta'$ depends only on $\gamma$ and not on $g$, but we do not need this for our purposes. See ~\cite[Prop 6.77]{SilvermanADS}.)

If $z\in\Per_n(f)\cap G/\Gamma$, the set $\pi^{-1}(z)$ is a $\Gamma$-orbit in $S$, so there is a bijection between $\Per_n(f)\cap G/\Gamma$ and the set of orbits $S/\Gamma$. Let $\Gamma_g$ be the subgroup of $\Gamma$ that fixes $g\in S$, and let $\delta$ be such that $\psi^n(g)=\delta(g)$. Then
\begin{equation*}
\#\{\gamma\in\Gamma: g\in\ker(\psi^n-\gamma)\} = \#\{\gamma\in\Gamma: g =\gamma^{-1}\delta(g)\}= |\Gamma_g|
\end{equation*}
By the orbit-stabilizer theorem ~\cite[Cor 4.10]{Isaacs},
\begin{align*}
\#S/\Gamma & = \sum_{g\in S}\frac{|\Gamma_g|}{|\Gamma|} = \frac{1}{|\Gamma|}\sum_{g\in S}\#\{\gamma\in\Gamma: g\in\ker(\psi^n-\gamma)\}\\
& = \frac{1}{|\Gamma|}\sum_{\gamma\in\Gamma}\#\{g\in S: g\in\ker(\psi^n-\gamma)\}=\frac{1}{|\Gamma|}\sum_{\gamma\in\Gamma}\#\ker(\psi^n-\gamma).
\end{align*}
Recall that $\psi(g)=\sigma(g)+h$. So $\psi^n(g)=\sigma^n(g)+h_n$ for some $h_n\in G$, and
\begin{equation*}
\ker(\psi^n-\gamma)=(\psi^n-\gamma)^{-1}(0)=(\sigma^n-\gamma)^{-1}(-h_n).
\end{equation*}
We assumed $\psi$ is surjective, so $\ker(\psi^n-\gamma)$ is nonempty. Therefore 
\begin{equation*}
\#(\sigma^n-\gamma)^{-1}(-h_n)=\#\ker(\sigma^n-\gamma),
\end{equation*}
completing the proof.
\end{proof}

\section{Maps from $\G_m$: Power Maps and Chebyshev Polynomials}\label{sec: G_m}
Let $\G_m$ be the multiplicative group. The endomorphism ring $\End(\G_m)$ is isomorphic to $\Z$, where the integer $d$ corresponds to the power map $x\mapsto x^d$. So every affine morphism $\psi:\G_m\to\G_m$ has the form $\psi(x)=ax^d$. The automorphism group is $\Aut(\G_m)\cong\Z^\times=\{\pm 1\}$. There are only two subgroups $\Gamma\subseteq\Aut(E)$: either $\Gamma$ is trivial or $\Gamma=\{x,x^{-1}\}$.

The underlying scheme of $\G_m$ is $\Spec k[x,x^{-1}]\cong\A^1\setminus\{0\}$, which is Zariski open in $\P^1$. If $\Gamma$ is trivial, then a power map arises from the following commutative diagram:
\[
\xymatrix{
 \G_m \ar[r]^{x\mapsto ax^d}\ar[d]^\wr & \G_m \ar[d]^\wr \\
\A^1\setminus\{0\}\ar [d]\ar[r] & \A^1\setminus\{0\}\ar[d]\\
\P^1 \ar[r]^f & \P^1
}
\]
There are many choices for the inclusion map $\A^1\setminus\{0\}\hookrightarrow\P^1$. The most obvious is the map that extends to the identity map on $\P^1$, in which case $f$ is the ``affine power map" $f(x)=ax^d$. Other choices give such maps conjugated by fractional linear transformations. Over an algebraically closed field, we can always conjugate by a linear polynomial $x\mapsto cx$ in order to make $f$ monic, so we may assume $f(x)=x^d$. (Recall that $\zeta$ is a conjugacy invariant.)

There are only two points in $\P^1$ that lie outside $\A^1\setminus\{0\}$, namely, 0 and $\infty$. If $d>0$, then $f$ fixes these two points, and if $d<0$, then $f$ swaps them. The group law of $\G_m$ is written multiplicatively, so if $d>0$, Lemma ~\ref{lem: periodic count} gives
\begin{equation}
\#\Per_n(f) = 2 + \#\ker(x^{d^n-1}),
\end{equation}
and if $d<0$, then
\begin{equation}
\#\Per_n(f) = \left\{
     \begin{array}{ll}
       2+\#\ker(x^{d^n-1}) & : d\text{ even}\\
       \#\ker(x^{d^n-1}) & : d\text{ odd}
     \end{array}
   \right.
\end{equation}

If we let $\Gamma=\{x,x^{-1}\}$, then $\G_m/\Gamma\cong\A^1$, and the quotient can be realized by the map $\pi(x)=x+x^{-1}$. There exists a polynomial $f$ such that the following diagram commutes ~\cite[Prop 6.6]{SilvermanADS}.
\[
\xymatrix{
 \G_m \ar[r]^{x\mapsto ax^d}\ar[d]^\pi & \G_m \ar[d]^\pi \\
\A^1\ar[d]\ar[r] & \A^1\ar[d]\\
\P^1 \ar[r]^f & \P^1
}
\]
If the inclusion $\A^1\hookrightarrow\P^1$ extends to the identity map and $a=1$, then $f$ is the $d$th Chebyshev polynomial $T_d(x)$ and satisfies 
\begin{equation*}
f(x+x^{-1})=x^d+x^{-d}.
\end{equation*}
As with power maps, choosing other inclusions or $a\neq 1$ simply results in fractional linear conjugates of Chebyshev polynomials. Because of the symmetry in the definition, positive and negative $d$ give rise to the same $f$. Here the count of Lemma ~\ref{lem: periodic count} is
\begin{equation}
\#\Per_n(f) = 1 + \frac{1}{2}\left(\#\ker(x^{d^n-1})+\#\ker(x^{d^n+1})\right).
\end{equation}

The kernel of the endomorphism $x\mapsto x^m$ is the set of $|m|$th roots of unity in $k$. If $(p,|m|)=1$, there are $|m|$ of these, and in general  
\begin{equation*}
\#\ker(x^m)=\frac{|m|}{p^{v_p(|m|)}}
\end{equation*}
because there are no nontrivial $p$th roots of unity.

Therefore, for a power map $f\in k(x)$ associated to the endomorphism $\sigma(x)=x^d$, if $d>0$ we have the formula
\begin{equation}\label{eqn: periodic count power map}
\#\Per_n(f) = 2 + \frac{d^n-1}{p^{v_p(d^n-1)}},
\end{equation}
and if $d<0$,
\begin{equation}\label{eqn: periodic count power map negative}
\#\Per_n(f) = \left\{
     \begin{array}{ll}
        2 + \frac{|d|^n-1}{p^{v_p(|d|^n-1)}} & : d\text{ even} \\
        \frac{|d|^n+1}{p^{v_p(|d|^n+1)}} & : d\text{ odd}
     \end{array}
   \right.
\end{equation}
For a Chebyshev polynomia $f$, the formula reads
\begin{equation}\label{eqn: periodic count Chebyshev polynomial}
\#\Per_n(f) = 1 + \frac{1}{2}\left(\frac{|d|^n-1}{p^{v_p(|d|^n-1)}}+\frac{|d|^n+1}{p^{v_p(|d|^n+1)}}\right).
\end{equation}

\section{Maps from $\G_a$: Additive and Subadditive Polynomials}\label{sec: G_a} Let $\G_a$ be the additive group. In characteristic zero, all endomorphisms of $\G_a$ are of the form $x\mapsto cx$, so $\End(\G_a)\cong k$. In positive characteristic, the Frobenius map $\phi(x)=x^p$ and its iterates are also endomorphisms, and $\End(\G_a)$ is the noncommutative polynomial ring $ k\langle\phi\rangle$ with the multiplication rule $\phi c=c^p\phi$ for $c\in k$.

The only automorphisms of $\G_a$ are the nonzero maps $x\mapsto cx$, as Frobenius is a bijection on $k$-valued points ($k$ is algebraically closed) but is not an isomorphism on the underlying scheme, which is $\A^1$. Therefore $\Aut(\G_a)\cong k^\times$. The finite subgroups $\Gamma\subseteq\Aut(\G_a)$ are all cyclic by a basic fact of field theory ~\cite[Lem 17.12]{Isaacs}, so there is some $d$ such that $\Gamma\cong\mu_d$, the group of $d$th roots of unity.

Let $\psi:\G_a\to\G_a$ be an affine morphism, that is, an element of $k\langle\phi\rangle$ composed with a translation. If $\Gamma=\{1\}$, then $\pi:\G_a\to G/\Gamma\cong\A^1$ is the identity morphism on the underlying scheme. There always exists an $f$ that fits into the following diagram:
\[
\xymatrix{
 \G_a \ar[r]^{\psi}\ar[d]^\pi & \G_a \ar[d]^\pi \\
\A^1\ar[d]\ar[r] & \A^1\ar[d]\\
\P^1 \ar[r]^f & \P^1
}
\]
There are many inclusions $\A^1\hookrightarrow\P^1$, but as with power maps, $f$ is determined up to conjugacy. Therefore we may assume that the inclusion extends to the identity, in which case $f$ fixes $\infty$ and is a polynomial. We call $f$ an additive polynomial, as $f(x+y)=f(x)+f(y)$ for all $x,y\in k$. 

If $\Gamma\cong\mu_d$ for $d>1$ and $(p,d)=1$, then the map $\pi:\G_a\to\G_a/\mu_d\cong \A^1$ can be taken to be $\pi(x)=x^d$. In this case there exists an $f$ to make the diagram commute if and only if $\psi$ satisfies $\psi(\omega_d x) = \omega_d\psi(x)$ for a primitive $d$th root of unity $\omega_d$. (This happens if and only if $\psi(x)$, written as a polynomial, has terms whose degrees are all 1 mod $d$.) If there is such an $f$, we call it a subadditive polynomial.

Let $\sigma\in k\langle\phi\rangle$ be an endomorphism of $\G_a$. The size of $\ker\sigma$ depends on the divisibility of $\sigma$ by $\phi$, that is,
\begin{equation*}
\#\ker\sigma = \frac{\deg\sigma}{p^{v_\phi(\sigma)}}.
\end{equation*}
Here $v_\phi(\sigma)$ denotes the largest power of the two-sided maximal ideal $(\phi)=\phi k\langle\phi\rangle$ that contains $\sigma$.

Let $\psi(x)=\sigma(x)+c$ for some $c\in\G_a$. For an additive or subadditive polynomial $f$, Lemma ~\ref{lem: periodic count} and the above observation yield
\begin{equation}\label{eqn: periodic count additive polynomial}
\#\Per_n(f) = 1 + \frac{1}{d}\sum_{\omega\in\mu_d}\#\ker(\sigma^n-\omega) = 1 + \frac{1}{d}\sum_{\omega\in\mu_d}\frac{(\deg\sigma)^n}{p^{v_\phi(\sigma^n-\omega)}}.
\end{equation}
Note that $\deg(\sigma^n-\omega)=\deg(\sigma^n)$ because $\omega:\G_a\to\G_a$ is the linear polynomial $\omega(x)=\omega x$, and $\deg\sigma=\deg\psi\geq 2$ by assumption.

\section{Maps from Elliptic Curves: Latt\`es Maps}\label{sec: Lattes} Let $E$ be an elliptic curve and let $\psi:E\to E$ be an affine morphism. The endomorphism ring $\End(E)$ can be identified with either $\Z$, an order in an imaginary quadratic field, or a maximal order in a quaternion algebra ~\cite[Thm V.3.1]{SilvermanEllipticCurves}. There are only six possibilities for $\Aut(E)$: it may be a cyclic group of order 2,3,4, or 6, or a certain nonabelian group of order 12 or 24 ~\cite[Thm III.10.1]{SilvermanEllipticCurves}. Let $\Gamma$ be a nontrivial subgroup of $\Aut(E)$. We say that $f$ is a Latt\`es map if the diagram commutes:
\[
\xymatrix{
E \ar[r]^{\psi}\ar[d]^\pi & E \ar[d]^\pi \\
E/\Gamma\ar[d]^\wr\ar[r] & E/\Gamma\ar[d]^\wr\\
\P^1 \ar[r]^f & \P^1
}
\]
As $E$ is projective, the curve $E/\Gamma$ is isomorphic to $\P^1$, unlike in the cases coming from $\G_m$ and $\G_a$. For a given choice of $\psi$ and $\Gamma$, there is not necessarily an $f$ that makes the diagram commute.

\begin{rem}
A Latt\`es map is often defined to be $f\in k(x)$ such that there exists a morphism $\psi:E\to E$ with $\deg\psi\geq 2$ and a finite separable cover $\pi:E\to\P^1$ such that the following diagram commutes. 
\[
\xymatrix{
E \ar[r]^{\psi}\ar[d]^\pi & E \ar[d]^\pi \\
\P^1 \ar[r]^f & \P^1
}
\]
This is equivalent to our definition. Any self-morphism of an elliptic curve can be written as the composition of an isogeny and a translation ~\cite[p 75]{SilvermanEllipticCurves}, so any morphism $\psi:E\to E$ with $\deg\psi\geq 2$ is affine. Also, if there exists a diagram as above, then there exists such a diagram with the same $f$ and a possibly different triple $(E',\psi',\pi')$ where the map $\pi'$ is the quotient of $E'$ by a subgroup of automorphisms. This result is due to Milnor over $\C$ ~\cite{MilnorLattes}, and Ghioca and Zieve in arbitrary characteristic ~\cite{GhiocaZieve}. For a sketch of the Ghioca-Zieve proof, see ~\cite[pp. 54-56]{SilvermanModuliSpaces}.
\end{rem}

By a general fact about morphisms of elliptic curves ~\cite[Thm III.4.10(a)]{SilvermanEllipticCurves},
\begin{equation*}
\#\ker\sigma = \deg_s\sigma = \frac{\deg\sigma}{\deg_i\sigma}.
\end{equation*}
Here $\deg_s$ and $\deg_i$ denote separable and inseparable degrees. In the rest of this section we develop a formula for $\deg_i(\sigma)$ in terms of the arithmetic of $\End(E)$.

First we set some notation. Let $N,\tr:\End(E)\otimes\Q\to\Q$ denote the norm and trace maps, or the reduced norm and trace in the case that $\End(E)\otimes\Q$ is a quaternion algebra. Let $\phi_m:E\to E^{(p^m)}$ be the $p^m$th power Frobenius morphism. For an isogeny $\sigma:E_1\to E_2$, write $\hat{\sigma}:E_2\to E_1$ for the dual isogeny. The $j$-invariant of $E$ is denoted by $j(E)$.

\begin{prop}\label{prop: j(E) transcendental}
Suppose $j(E)$ is transcendental over $\F_p$. Let $\sigma\in\End(E)$. Then $\sigma\in\Z$ and 
\begin{equation*}
\deg_i(\sigma)=p^{v_p(\sigma)}.
\end{equation*}
\end{prop}
\begin{proof}
If $j(E)\notin\overline{\F}_p$, then $\End(E)\cong\Z$ ~\cite[p 145]{SilvermanEllipticCurves}. The multiplication by $p$ map $[p]:E\to E$ has inseparable degree $p$, because $[p]=\hat{\phi_1}\circ\phi_1$ and the dual isogeny $\hat{\phi_1}$ is separable ~\cite[Thm V.3.1]{SilvermanEllipticCurves}. The multiplication by $m$ map is separable if $(p,m)=1$. Therefore the isogeny $[m]:E\to E$ has inseparable degree equal to $p^{v_p(m)}$.
\end{proof}

For the rest of this section, suppose that $j(E)$ is algebraic over $\F_p$. Up to isomorphism, $E$ is defined over a finite field, so the ring $\End(E)$ can be identified with an order in an imaginary quadratic field if $E$ is ordinary or a maximal order in a quaternion algebra if $E$ is supersingular ~\cite[Thm V.3.1]{SilvermanEllipticCurves}.

\begin{prop}\label{prop: inseparable degree ordinary}
Let $E$ be ordinary and let $K=\End(E)\otimes\Q$. Let $\sigma\in\End(E)$. Let $\mathfrak{p}$ be the extension to $\mathcal{O}_K$ of the ideal in $\End(E)$ consisting of all inseparable isogenies. Then $\mathfrak{p}$ is prime and
\begin{equation*}
\deg_i(\sigma)=p^{v_\mathfrak{p}(\sigma)}.
\end{equation*}
\end{prop}

\begin{proof}
By ~\cite[Cor II.2.12]{SilvermanEllipticCurves} we can write $\sigma=\lambda\circ\phi_{m}$ where $p^m=\deg_i(\sigma)$ and $\lambda:E^{(p^{m})}\to E$ is separable. It follows that $\deg_i(\sigma)\geq p^m$ if and only if $\sigma$ factors through $\phi_{m}:E\to E^{(p^{m})}$. (If $\sigma=0$, we set $\deg_i(\sigma)=\infty$.)

We know that $\End(E)$ is an order in $\mathcal{O}_K$ for some imaginary quadratic field $K$, and the conductor of $\End(E)$ is prime to $p$ ~\cite{Deuring}. Let $I_m$ be the $m$th \emph{inseparable ideal} of $\End(E)$, defined as follows:
\begin{equation*}
I_m = \{\sigma\in\End(E): \deg_i(\sigma)\geq p^m\}.
\end{equation*}
It is routine to show that $I_m$ is an ideal. If $\sigma\tau\in I_m$, then 
\begin{equation*}
\deg_i(\sigma\tau)=\deg_i(\sigma)\deg_i(\tau)\geq p^m.
\end{equation*}
If $\sigma\notin I_m$ then necessarily $\deg_i(\tau)\geq p$, so $\deg_i(\tau^m)\geq p^m$ and $\tau^m\in I_m$. For $m=1$ this shows that $I_1$ is prime, and for $m>1$ that $I_m$ is $I_1$-primary.

If $\End(E)$ were a Dedekind domain, it would follow that $I_m=(I_1)^m$ for each $m>1$, but orders of Dedekind domains are not in general Dedekind domains. Instead, consider the integral extension $\mathcal{O}_K/\End(E)$ and the prime ideal $\mathfrak{p}$ of $\mathcal{O}_K$ lying over $I_1$. The multiplication by $p$ map $[p]:E\to E$ is inseparable, so $p\in I_1$, and therefore $p\mathcal{O}_K\subseteq I_1\mathcal{O}_K\subseteq\mathfrak{p}$. So $I_1\mathcal{O}_K$ is either $\mathfrak{p}$ or $p\mathcal{O}_K=\mathfrak{p}\hat{\mathfrak{p}}$ (as $p$ splits in $\mathcal{O}_K$ ~\cite{Deuring}).

This shows that the ideal $I_1\mathcal{O}_K$ is prime to the conductor of $\End(E)$. For ideals in an order that are prime to the conductor, extension to $\mathcal{O}_K$ and contraction are inverses, and unique factorization holds ~\cite[Prop 7.20]{Cox}. Therefore $I_1\mathcal{O}_K=\mathfrak{p}$ and the $I_1$-primary ideal $I_m$ equals $(I_1)^m$ for each $m>1$. It follows easily that $(I_1)^m\mathcal{O}_K=(I_1\mathcal{O}_k)^m=\mathfrak{p}^m$. We conclude that $\deg_i(\sigma)=p^{v_\mathfrak{p}(\sigma)}$.
\end{proof}

\begin{rem}
If $E$ is defined over $\F_p$, the $p$th power Frobenius morphism $\phi_1$ is an element of $\End(E)$. In this case, the ideal $\mathfrak{p}$ in Proposition ~\ref{prop: inseparable degree ordinary} is the principal ideal $\phi_1\mathcal{O}_K$, and $v_\mathfrak{p}(\sigma)$ is simply the highest power of $\phi_1$ that divides $\sigma$ in $\mathcal{O}_K$. In general, the ideal $\mathfrak{p}$ need not be principal.
\end{rem}

\begin{prop}\label{prop: inseparable degree supersingular}
Let $E$ be supersingular, so that $\End(E)$ can be identified with a maximal order $\mathcal{O}$ of the quaternion algebra $B$. There exists a two-sided maximal ideal $I$ of $\mathcal{O}$ such that
\begin{equation*}
\deg_i(\sigma)=p^{v_I(\sigma)}.
\end{equation*}
\end{prop}

\begin{proof}
If we write $\sigma=\lambda\circ\phi_m$ where $\lambda$ is separable, then $\deg(\lambda)$ is not divisible by $p$. If it were, the map $\lambda\circ\hat{\lambda}=[N(\lambda)]=[\deg(\lambda)]:E\to E$ would factor through $[p]:E\to E$ and would be inseparable, so one of $\lambda$ or $\hat{\lambda}$ would be inseparable. If $\hat{\lambda}$ were inseparable it would factor through $\phi_1$, so $\lambda$ would factor through $\hat{\phi_1}$, which is inseparable ~\cite[Thm V.3.1]{SilvermanEllipticCurves}, contradicting the fact that $\lambda$ is separable. Moreover, $\deg\phi_m$ is a $p$-power, and is therefore the largest power of $p$ that divides $\deg\sigma$. This shows that
\begin{equation*}
\deg_i\sigma=\deg\phi_m=p^{v_p(\deg\sigma)}=p^{v_p(N(\sigma))}.
\end{equation*}

We have $\End(E)\cong\mathcal{O}$, which is a maximal order of $B$, the unique quaternion algebra over $\Q$ ramified exactly at $p$ and $\infty$ ~\cite{Deuring}. Consider the localization $\mathcal{O}_p=\mathcal{O}\otimes\Z_p$, which is the unique maximal order of $B_p=B\otimes_\Q \Q_p$, and the inclusion $\mathcal{O}\hookrightarrow\mathcal{O}_p$. 

In $\mathcal{O}_p$ there is a uniformizing element $\pi$ such that $\pi\mathcal{O}_p$ is the unique two-sided maximal ideal of $\mathcal{O}_p$, and every ideal of $\mathcal{O}_p$ is a power of $\pi\mathcal{O}_p$ ~\cite[Ch 2, Thm 1.3]{Vigneras}. In particular, $p\mathcal{O}_p=\pi^2\mathcal{O}_p$, so $v_{\pi\mathcal{O}_p}(p)=2$. Let $I=\pi\mathcal{O}_p\cap\mathcal{O}$. The ideal $I$ is maximal in $\mathcal{O}$ because locally it is either maximal or the unit ideal: $I_p=\pi\mathcal{O}_p$,  and $I_\ell=\mathcal{O}_\ell$ for $\ell\neq p$ (this is because $p$, which lies in $I$, is invertible in $\mathcal{O}_\ell$). The ideal $I$ is also two-sided because it is two-sided locally ~\cite[p. 84]{Vigneras}. Therefore $v_I$ is a valuation on $\mathcal{O}$.

Any supersingular $E$ is defined over $\F_{p^2}$, and there exists an automorphism $i:E\to E$ such that $\phi_2=i\circ[p]$. As $v_I(p)=2$, it follows that $v_I(\sigma)=v_p(N(\sigma))$ for all $\sigma\in\mathcal{O}$ such that $v_p(N(\sigma))$ is even, i.e. such that $\sigma=\lambda\circ[p^n]$ for some separable $\lambda$. If $v_p(N(\sigma))$ is odd then $v_p(N(\sigma^2))$ is even, so $v_I(\sigma^2)=v_p(N(\sigma^2))$, and therefore $v_I(\sigma)=v_p(N(\sigma))$. 
\end{proof}

Let $\psi:E\to E$ be written $\psi(x) = \sigma(x) + P$, where $\sigma\in\End(E)$ and $P\in E$. The above propositions together with Lemma ~\ref{lem: periodic count} prove the following formulas. If $j(E)$ is transcendental, then necessarily $\Gamma=\{\pm 1\}$, and
\begin{equation}\label{eqn: Lattes transcendental}
\#\Per_n(f) = \frac{1}{2}\left(\frac{|\sigma^n-1|}{p^{v_p(\sigma^n-1)}} + \frac{|\sigma^n+1|}{p^{v_p(\sigma^n+1)}}\right).
\end{equation}
If $j(E)$ is algebraic and $E$ is ordinary, then there exists $\mathfrak{p}$ such that
\begin{equation}\label{eqn: Lattes ordinary}
\#\Per_n(f) = \frac{1}{|\Gamma|}\sum_{\gamma\in\Gamma}\frac{N(\sigma^n-\gamma)}{p^{v_{\mathfrak{p}}(\sigma^n-\gamma)}}.
\end{equation}
If $j(E)$ is algebraic and $E$ is supersingular, then there exists $I$ such that
\begin{equation}\label{eqn: Lattes supersingular}
\#\Per_n(f) = \frac{1}{|\Gamma|}\sum_{\gamma\in\Gamma}\frac{N(\sigma^n-\gamma)}{p^{v_I(\sigma^n-\gamma)}}.
\end{equation}

The sequence $n\mapsto N(\sigma^n-\gamma)$ that appears in the above equations satisfies a linear recurrence relation and is therefore periodic when reduced mod any prime $\ell$. We record for future reference the next proposition, which determines its possible periods.

\begin{prop}\label{prop: norm recurrent}
Let $j(E)\in\overline{\F}_p$ and  let $a_n=N(\sigma^n-\gamma)$, where $\sigma,\gamma\in\End(E)$. For any prime $\ell$, the sequence $(a_n\pmod{\ell})$ is periodic of period dividing $(\ell-1)(\ell^2-1)\ell^A$ for some integer $A$.
\end{prop}
\begin{proof}
Let $T=\tr(\sigma)$ and $N=N(\sigma)$. We compute
\begin{align*}
a_n=\widehat{(\sigma^n-\gamma)}(\sigma^n-\gamma)& = \hat{\sigma}^n\sigma^n - \sigma^n\hat{\gamma} - \hat{\sigma}^n\gamma + \hat{\gamma}\gamma = N^n - \tr(\sigma^n\hat{\gamma}) + N(\gamma).
\end{align*}
Let $b_n=N^n$, $c_n=\tr(\sigma^n\hat{\gamma})$, and $d_n=N(\gamma)$. Certainly $b_n=Nb_{n-1}$ and $d_n=d_{n-1}$; this shows that the linearly recurrent sequences $(b_n)$ and $(d_n)$ have characteristic polynomials $x-N$ and $x-1$ in the sense of ~\cite[Ch 6]{LidlNiederreiter}.

Whether $\End(E)$ is an order in an imaginary quadratic field or an order in a quaternion algebra, any $\sigma\in\End(E)$ satisfies the Cayley-Hamilton identity $\sigma^2-T\sigma+N=0$. For $n\geq 2$,
\begin{align*}
c_n - T c_{n-1} + N c_{n-2} & = \tr(\sigma^n\hat{\gamma}) -T\tr(\sigma^{n-1}\hat{\gamma}) + N\tr(\sigma^{n-2}\hat{\gamma})\\
 & =  \tr((\sigma^2-T\sigma+N)\sigma^{n-2}\hat{\gamma}) = \tr(0)=0.
\end{align*}
Therefore $(c_n)$ is a linearly recurrent sequence, and its characteristic polynomial is $x^2-Tx+N$. It follows from ~\cite[Thm 6.55]{LidlNiederreiter} that $(a_n)$ is linearly recurrent with characteristic polynomial equal to
\begin{equation*}
g(x)=(x-1)(x-N)(x^2-Tx+N)\in\F_\ell[x].
\end{equation*}
Recall that if $g(0)\neq 0$, $\ord g(x)$ is defined to be the least $n$ such that $g(x)$ divides $x^n-1$. By ~\cite[Thm 6.27]{LidlNiederreiter}, the least period of $(a_n\pmod{\ell})$ divides $\ord g(x)$, and by ~\cite[Thm 3.11]{LidlNiederreiter}, there is some $A\geq 0$ such that
\begin{equation*}
\ord g(x)=LCM[\ord(x-1),\ord (x-N),\ord(x^2 - Tx+ N)]\ell^A.
\end{equation*}
The integer $A$ reflects the possible presence of repeated factors of $g(x)$. If $x-1$, $x-N$, and $x^2-Tx+N$ are coprime in $\F_\ell[x]$, then $A=0$. 
\end{proof}

\section{Lifting the Exponent}\label{sec: lifting the exponent}
The periodic point counts established in the previous sections all contain an expression of the form $v_P(x^n-\gamma)$, where $P$ is a prime ideal of an endomorphism ring $R$. In this section we develop formulas for writing these expressions in terms of $v_p(n)$. These resemble a result in elementary number theory popularly known as ``lifting the exponent", which is related to (but does not follow from) Hensel's Lemma.

\begin{lem}\label{lem: Lifting the Exponent Number Field}
Let $K$ be a number field. Let $\mathfrak{p}$ be a prime of $\mathcal{O}_K$ lying over the rational prime $p$, and let $e$ be the ramification index of $\mathfrak{p}$ over $p$. Let $x,y\in\mathcal{O}_K$ be such that $x,y\notin\mathfrak{p}$ and $x-y\in\mathfrak{p}$. If $e+1\geq p-1$, further assume that $v_\mathfrak{p}(x-y)\geq\frac{e+1}{p-1}$. Then
\begin{equation*}
v_\mathfrak{p}(x^n-y^n) = v_\mathfrak{p}(x-y) + ev_p(n).
\end{equation*}
\end{lem}
\begin{proof}
The proof is by induction on $v_p(n)$. Assume that $x\neq y$; otherwise the proposition holds trivially.

First suppose that $v_p(n)=0$, which guarantees $v_\mathfrak{p}(n)=0$. We compute
\begin{equation*}
v_\mathfrak{p}(x^n-y^n) =v_\mathfrak{p}(x-y) + v_\mathfrak{p}(x^{n-1}+x^{n-2}y+\dots+xy^{n-2}+y^{n-1}).
\end{equation*}
As $x- y\in\mathfrak{p}$, we have $x^{n-1}+x^{n-2}y+\dots+xy^{n-2}+y^{n-1}\equiv nx^n\pmod{\mathfrak{p}}$, and $nx^n\notin\mathfrak{p}$. Therefore $v_\mathfrak{p}(x^n-y^n) = v_\mathfrak{p}(x-y)$, proving the proposition in this case.

If we show that the proposition holds for $n=p$, it follows for all $n$ by induction. Let $v=v_\mathfrak{p}(x-y)$, so that $x=y+z$ for some $z\in\mathfrak{p}^v\setminus\mathfrak{p}^{v+1}$. By the binomial theorem,
\begin{equation*}
x^p = \sum_{i=0}^p {p\choose i}z^{i}y^{p-i}\equiv y^n + pzy^{n-1} \pmod{\mathfrak{p}^{v+e+1}}.
\end{equation*}
For $i\geq 2$, the $i$th term of the expansion is in $\mathfrak{p}^{v+e+1}$. If $i\neq p$ this is because $p$ divides ${p\choose i}$, so ${p\choose i}\in\mathfrak{p}^e$ and ${p\choose i}z^iy^{n-i}$ lies in $\mathfrak{p}^{e+iv}\subseteq\mathfrak{p}^{v+e+1}$, as $v\geq 1$. If $i=p$, this is because we assumed that $v=v_\mathfrak{p}(z)\geq\frac{e+1}{p-1}$, so $v_\mathfrak{p}(z^p)=pv\geq v+e+1$ and $z^p\in\mathfrak{p}^{v+e+1}$. The $i=1$ term is not in $\mathfrak{p}^{v+e+1}$, as $v_\mathfrak{p}(pz)=e+1$. Therefore $x^p-y^p\in\mathfrak{p}^{v+e}\setminus\mathfrak{p}^{v+e+1}$, so $v_\mathfrak{p}(x^p-y^p)=v+e$ and we are done.
\end{proof}

\begin{lem}\label{lem: Lifting the Exponent Quaternion Algebra}
Let $p$ be a prime. Let $B$ be the unique quaternion algebra over $\Q$ ramified precisely at $p$ and $\infty$, and let $\mathcal{O}$ be a maximal order of $B$. Let $\pi$ be a uniformizer for $\mathcal{O}_p$, and let $I=\pi\mathcal{O}_p\cap\mathcal{O}$. Let $x,y\in\mathcal{O}$ be such that $x,y\notin I$ and $x-y\in I$. If $p=3$, assume further that $v_I(x-y)\geq 2$, and if $p=2$, assume further that $v_I(x-y)\geq 3$. Then
\begin{equation*}
v_I(x^n-y^n) = v_I(x-y) + 2v_p(n).
\end{equation*}
\end{lem} 

\begin{proof}
The proof is essentially the same as the proof of Lemma ~\ref{lem: Lifting the Exponent Number Field}, and is omitted.
\end{proof}

\begin{lem}\label{lem: Lifting the Exponent G_a}
Let $k$ be an algebraically closed field of characteristic $p$, and let $k\langle\phi\rangle$ be the noncommutative polynomial ring with the multiplication rule $\phi c = c^p\phi$ for $c\in k$. Let $x\in k\langle\phi\rangle$ be such that $x-1\in\phi k\langle\phi\rangle$. Then
\begin{equation*}
v_\phi(x^n-1) = v_\phi(x-1)p^{v_p(n)}.
\end{equation*}
\end{lem}
\begin{proof}
First assume that $v_p(n)=0$. Then
\begin{align*}
x^n-1 & = (1+(x-1))^n-1 \equiv na^{n-1}(x-1)\pmod{\phi^2 k\langle\phi\rangle},
\end{align*}
and $v_\phi(x^n-1)=v_\phi(x-1)$.

Next let $n=p$. As $x^p-1=(x-1)^p$, we have $v_\phi(x^p-1)=pv_\phi(x-1)$. The proposition follows by induction on $v_p(n)$.
\end{proof}

\section{Background from Automatic Sequences}\label{sec: Automatic Sequences}

This section contains several results from the theory of automatic sequences. A sequence $(a_n)$ is $k$-automatic if it can be produced as the output of a deterministic finite automaton that takes as input the base-$k$ expansion of the integer $n$. The theorems in this section are stated so that they can be used later without any specific knowledge of finite automata or automatic sequences. A good introduction to the theory can be found in ~\cite{AlloucheShallit}.

The following two theorems underlie our proof of transcendence. Christol's theorem gives a correspondence between automatic sequences and algebraic power series, and Cobham's theorem shows that only eventually periodic sequences can be automatic with respect to multiplicatively independent bases. 

\begin{thm}[Christol]
The formal power series $\sum_{n=0}^\infty a_n t^n\in\F_p[[t]]$ is algebraic over $\F_p(t)$ if and only if the coefficient sequence $(a_n)$ is $p$-automatic.
\end{thm}
\begin{proof}
~\cite[Thm 12.2.5]{AlloucheShallit}.
\end{proof}

\begin{thm}[Cobham]
Let $p$ and $q$ be multiplicatively independent positive integers (i.e. $\log p/\log q\notin\Q$). If the sequence $(a_n)$ is both $p$-automatic and $q$-automatic, then it is eventually periodic.
\end{thm}
\begin{proof}
~\cite[Thm 11.2.2]{AlloucheShallit}.
\end{proof}

The converse to Cobham's theorem is also true.
\begin{thm}
Let $(a_n)$ be eventually periodic. Then $(a_n)$ is $k$-automatic for every positive integer $k$. 
\end{thm}
\begin{proof}
~\cite[Thm 5.4.2]{AlloucheShallit}.
\end{proof}

The following is a corollary to Christol's theorem that will be used to derive the contradiction that shows that $\zeta_f$ is transcendental.
\begin{cor}\label{cor: ChristolCor}
If $\sum_{n=0}^\infty a_n t^n\in\Z[[t]]$ is algebraic over $\Q(t)$, then the reduced sequence $((a_n) \bmod{p})$ is $p$-automatic for every prime $p$.
\end{cor}
\begin{proof}
~\cite[Thm 12.6.1]{AlloucheShallit}.
\end{proof}

The next two propositions shows that the set of $p$-automatic sequences over a ring is closed under both the pointwise application of arithmetic operations and the operation of extracting subsequences indexed by arithmetic progressions.

\begin{prop}\label{prop: automatic closure}
Let $(a_n)$ and $(b_n)$ be $p$-automatic sequences with entries in the ring $R$, and let $c\in R$. The sequences $(a_n+b_n)$, $(a_nb_n)$, and $(ca_n)$ are $p$-automatic, as is the sequence $(a_n^{-1})$ if each $a_n$ is invertible. Also, the subsequence $(a_{mn+b})$ is $p$-automatic for any $m,b\in\Z_+$.
\end{prop}
\begin{proof}
The closure properties under arithmetic operations are special cases of the general theorem that the set of $p$-automatic sequences with entries in the set $\Delta$ is closed under the pointwise application of any binary operation $(\cdot,\cdot):\Delta\times\Delta\to\Delta$ ~\cite[Cor 5.4.5]{AlloucheShallit} and the completely trivial theorem that it is closed under the pointwise application of any unary operation $(\cdot):\Delta\to\Delta$ (this follows directly from the definition of an automatic sequence). For the claim about the subsequences $(a_{mn+b})$, see ~\cite[Thm 6.8.1]{AlloucheShallit}.
\end{proof}

Propositions ~\ref{prop: v_p automaticity} and ~\ref{prop: v_p automaticity G_a} are the major technical results of this section. They will be needed to produce a contradiction at key moments in the proof of Theorems ~\ref{thm: main} and ~\ref{thm: main2}. 

\begin{prop}\label{prop: v_p automaticity}
Let $p$ and $\ell$ be distinct primes. Suppose $a\in\Z_+$, $a\not\equiv 1\pmod{\ell}$, and $(a,\ell)=1$. Also suppose $\alpha,\beta\in\Z$, $\alpha\neq 0$, and $v_p(\alpha)\leq v_p(\beta)$. Let the sequence $(a_n)$ with entries in $\Z/\ell\Z$ be defined by
\begin{equation*}
a_n=a^{v_p(\alpha n+\beta)}\pmod{\ell}.
\end{equation*}
The sequence $(a_n)$ is not $\ell$-automatic.
\end{prop}
\begin{proof}
Let $d$ be the multiplicative order of $a$ mod $\ell$. It follows from the assumptions that $d$ exists and is greater than 1. The sequence $n\mapsto a^{v_p(n)}$ is a function of the equivalence class of $v_p(n)$ mod $d$, so it is $p$-automatic by ~\cite[Lem 6]{BridyActa}. Therefore the sequence $(a_n)$ is $p$-automatic by Proposition ~\ref{prop: automatic closure}.

Assume by way of contradiction that $(a_n)$ is $\ell$-automatic. Distinct primes are multiplicatively independent, so by Cobham's theorem, $(a_n)$ is eventually periodic. Let $c$ be its eventual period, so that $a_{n+xc}=a_n$ for sufficiently large $n$ and every positive $x$. This means that 
\begin{equation}
a^{v_p(\alpha n+\beta)}\equiv a^{v_p(\alpha(n+xc)+\beta)}\pmod{\ell},
\end{equation}
which implies that 
\begin{equation}
v_p(\alpha n+\beta)\equiv v_p(\alpha(n+xc)+\beta)\pmod{d}.
\end{equation}

Let $\alpha'=\alpha/p^{v_p(\alpha)}$ and $\beta'=\beta/p^{v_p(\alpha)}$. It is clear that $(p,\alpha')=1$, and it follows from our assumption that $v_p(\alpha)\leq v_p(\beta)$ that $\beta'$ is an integer. We have
\begin{equation}\label{eqn: contradiction}
v_p(\alpha' n+\beta')\equiv v_p(\alpha'(n+xc)+\beta')\pmod{d}.
\end{equation}
Let $m=v_p(c)$ so that $c'=c/p^m$ and $(p,c')=1$. We can solve the congruence
\begin{equation}\label{eqn: one}
\alpha'n\equiv -\beta' +p^m\pmod{p^{m+2}}
\end{equation}
for $n$, and choose such an $n$ to be large enough so that the sequence $(a_n)$ is periodic at $n$. Therefore $v_p(\alpha' n+\beta')=m$. We can also solve
\begin{equation}\label{eqn: two}
\alpha' c'x\equiv p-1\pmod{p^{m+2}}
\end{equation}
for $x$, and choose such an $x$ to be positive. Adding Equation ~\ref{eqn: one} and $p^m$ times Equation ~\ref{eqn: two} gives
\begin{equation}
\alpha'(n+xc)\equiv -\beta'+ p^{m+1}\pmod{p^{m+2}}
\end{equation}
and therefore $v_p(\alpha'(n+xc)+\beta')=m+1$. So by Equation ~\ref{eqn: contradiction},
\begin{equation}
m\equiv m+1\pmod{d}.
\end{equation}
As $d>1$, this is a contradiction.
\end{proof}

\begin{prop}\label{prop: v_p automaticity G_a}
Let $a\in\Z_+$ and let $p$ and $\ell$ be primes with $\ell>p^{ap^a}$. If $p$ is odd, also assume that $(p,\ell-1)=1$. If $p=2$, instead assume that $\ell\equiv 7\pmod{8}$. Let the sequence $(a_n)$ with entries in $\Z/\ell\Z$ be defined by
\begin{equation*}
a_n=p^{ap^{v_p(n)}}\pmod{\ell}.
\end{equation*}
The sequence $(a_n)$ is not $\ell$-automatic.
\end{prop}

\begin{proof}
Let $d$ be the multiplicative order of $p^a$ mod $\ell$. Since $\ell>p^{ap^a}$, we have $d>p^a$. The sequence $a_n$ is a function of the equivalence class of $p^{v_p(n)}$ mod d.

First assume that $p$ is odd. Then $(p^a,\ell-1)=1$, and as $d$ divides $\ell-1$, also $(p^a,d)=1$. Let $e$ be the multiplicative order of $p^a$ mod $d$, and note that $e\geq 2$. So $a_n=a_m$ iff $p^{v_p(n)}\equiv p^{v_p(m)}\pmod{d}$ iff $v_p(n)\equiv v_p(m)\pmod{e}$. In particular, $a_n$ is a function of the equivalence class of $v_p(n)$ mod $e$. By ~\cite[Lem 6]{BridyActa}, $(a_n)$ is $p$-automatic.

If instead $p=2$, then $\ell\equiv 7\pmod{8}$, so 2 is a quadratic residue mod $\ell$. It follows that $d$ divides $\frac{\ell-1}{2}$, so $d$ is odd and $(p,d)=1$. Let $e$ be the multiplicative order of $p$ mod $d$, and again note that $e\geq 2$, so again $a_n=a_m$ if and only if $v_p(n)\equiv v_p(m)\pmod{e}$. In particular, $(a_n)$ is $p$-automatic.

Assume that $(a_n)$ is $\ell$-automatic. By Cobham's theorem, $(a_n)$ is eventually periodic. Let $c$ be its eventual period. Therefore, for $n$ large and $x>0$,
\begin{equation}
v_p(n+xc)\equiv v_p(n)\pmod{e}.
\end{equation}
This is impossible by the proof of Proposition ~\ref{prop: v_p automaticity}; simply set $\alpha=1$ and $\beta=0$.
\end{proof}

\section{Proof of Theorem 1.2}\label{sec: Proof 1.2}

Let $k$ be an algebraically closed field of characteristic $p$, and let $f\in k(x)$ be a separable power map, Chebyshev polynomial, or Latt\`es map. Let $\zeta=\zeta(f,\P^1(k);t)$. Assume by way of contradiction that $\zeta$ is algebraic over $\Q(t)$. Its derivative $\zeta'$ is algebraic, so its logarithmic derivative $\zeta'/\zeta$ is algebraic. We compute
\begin{equation*}
\zeta'/\zeta=(\log\zeta)'=\sum_{n\geq 1}\#\Per_n(f)t^{n-1}.
\end{equation*}
 By Corollary ~\ref{cor: ChristolCor}, for any prime $\ell$, the reduction of the sequence $(\#\Per_n(f))$ mod $\ell$\ is $\ell$-automatic. By carefully choosing $\ell$, we will produce a contradiction, showing that $\zeta$ is transcendental.

\subsection{Power Maps and Chebyshev Polynomials}
Let $f$ be the power map $f(x)=x^d$, and note that $(p,d)=1$ because $f$ is separable. Let $m$ be even and such that $d^m\equiv 1\pmod{p}$. Let $\ell>p$ be a prime to be determined, and consider the sequence $(a_n)$ with entries in $\F_\ell$ given by
\begin{equation*}
a_n = \#\Per_{mn}(f)\pmod{\ell}.
\end{equation*}
By Proposition ~\ref{prop: automatic closure}, $(a_n)$ is $\ell$-automatic. As $mn$ is even, Equation ~\ref{eqn: periodic count power map} and Proposition ~\ref{lem: Lifting the Exponent Number Field} give
\begin{equation*}
a_n = 2+\frac{d^{mn}-1}{p^{v_p(d^{mn}-1)}}=2+(d^{mn}-1)(p^{-1})^{v_p(d^{m}-1)+v_p(n)}.
\end{equation*}

First suppose $p$ is odd. Note that $(d^{mn}-1)=d^m-1$ when $n\equiv 1\pmod{\ell-1}$, and that $d^m-1\not\equiv 0\pmod{\ell}$. Consider the subsequence 
\begin{equation*}
b_n=(d^m-1)(p^{v_p(d^m-1)})^{-1}(a_{(\ell-1)n+1}-2)^{-1},
\end{equation*}
which is $\ell$-automatic by proposition ~\ref{prop: automatic closure}. Then
\begin{equation*}
b_n= p^{v_p((\ell-1)n+1)}.
\end{equation*}
Choose $\ell$ such that $\ell\equiv 2\pmod{p}$. Then $v_p(\ell-1)=0=v_p(1)$, and by Proposition ~\ref{prop: v_p automaticity}, $(b_n)$ is not automatic, which is a contradiction.

Now suppose $p=2$. Let $m=2$, so that $d^m\equiv 1\pmod{4}$ (the separability of $f$ forces $d$ to be odd). Let $b_n$ be given by
\begin{equation*}
b_n=(d^{2}-1)(p^{v_p(d^2-1)})^{-1}(a_{(\ell-1)n+2}-2)^{-1},
\end{equation*}
so that
\begin{equation*}
b_n= p^{v_p((\ell-1)n+2)}.
\end{equation*}
Choose $\ell$ such that $\ell\equiv 3\pmod{4}$ and $(\ell,d^2-1)=1$. Then $v_p(\ell-1)=1=v_p(2)$, and again Proposition ~\ref{prop: v_p automaticity} gives a contradiction.

Let $f$ be the $d$th Chebyshev polynomial. Again it must be true that $(p,d)=1$, because otherwise the $\psi(x)$ that fits into Diagram ~\ref{eqn: dynamically affine} factors through the $p$th power map and is inseparable, so $f$ is also inseparable. Again let $m$ be such that $d^m\equiv 1\pmod{p}$, and let $a_n = \#\Per_{mn}(f)\pmod{\ell}$ for some $\ell$ to be determined. By equation ~\ref{eqn: periodic count Chebyshev polynomial} and Proposition ~\ref{lem: Lifting the Exponent Number Field},
\begin{align*}
a_n & = 1+\frac{1}{2}\left(\frac{d^{mn}-1}{p^{v_p(d^{mn}-1)}}+\frac{d^{mn}+1}{p^{v_p(d^{mn}+1)}}\right)\\
& =1+\frac{1}{2}\left((d^{mn}-1)(p^{-1})^{v_p(d^{m}-1)+v_p(n)}+(d^{mn}+1)(p^{-1})^{v_p(d^{mn}+1)}\right) .
\end{align*}

First suppose $p$ is odd. Then $v_p(d^{mn}+1)=0$ because $v_p(d^{mn}-1)>0$ by Proposition ~\ref{lem: Lifting the Exponent Number Field}. The sequence
\begin{equation*}
n\mapsto(d^{mn}+1)
\end{equation*}
is periodic, and so is $\ell$-automatic. Let $b_n$ be defined by
\begin{equation*}
b_n=(2(a_{(\ell-1)n+1}-1)-(d^{m(\ell-1)n+1}))(d^m-1)^{-1}(p^{v_p(d^m-1)}).
\end{equation*}
As before, $b_n$ is $\ell$-automatic, but
\begin{equation*}
b_n^{-1}= p^{v_p((\ell-1)n+1)}.
\end{equation*}
Choosing $\ell$ such that $\ell \equiv 2\pmod{p}$ gives a contradiction by Proposition ~\ref{prop: v_p automaticity}.

If $p=2$, then let $m=2$. In this case, $d^{2n}\equiv 1\pmod{4}$ for all $n$, so $v_p(d^{2n}+1)=1$. The sequence
\begin{equation*}
n\mapsto(d^{2n}+1)
\end{equation*}
is still eventually periodic, so using the same manipulations as above, the sequence
\begin{equation*}
b_n^{-1}= p^{v_p((\ell-1)n+2)}
\end{equation*}
is $\ell$-automatic. If we pick $\ell$ such that $\ell\equiv 3\pmod{4}$ and $(\ell,d^2-1)=1$, then $v_p(\ell-1)=1$ and again this is a contradiction by Proposition ~\ref{prop: v_p automaticity}.

\subsection{Latt\`es Maps} Let $f$ be a Latt\`es map associated to the elliptic curve $E$. If $j(E)$ is transcendental over $\overline{\F}_p$, then $\sigma\in\Z$, and as $f$ is separable we have $(p,\sigma)=1$ (otherwise $\sigma$ would be inseparable). By equation ~\ref{eqn: Lattes transcendental},
\begin{equation*}
\#\Per_n(f) = \frac{1}{2}\left(\frac{|\sigma^n-1|}{p^{v_p(\sigma^n-1)}} + \frac{|\sigma^n+1|}{p^{v_p(\sigma^n+1)}}\right).
\end{equation*}
If $\sigma>0$, then this is the same as the periodic point count for the degree $|\sigma|$ Chebyshev polynomial $T_{|\sigma|}$, and we have already shown that there exists an $\ell$ such that $\#\Per_n(f)\pmod{\ell}$ is not $\ell$-automatic. In fact, the argument also goes through if $\sigma<0$, as we simply choose an even $m$ and use the subsequence $\#\Per_{mn}(f)\pmod{\ell}$, in which case $\sigma^{mn}$ is always positive.

Now suppose $j(E)\in\overline{\F}_p$ and $E$ is ordinary. The ring $\End(E)$ is an order in an imaginary quadratic field $K$, and $\Aut(E)$ is cyclic of order 2,4, or 6. Therefore $\Gamma$ is isomorphic to one of $\mu_2,\mu_3,\mu_4$, or $\mu_6$. By Proposition ~\ref{prop: inseparable degree ordinary}, there is a prime $\mathfrak{p}$ of $\mathcal{O}_K$ such that
\begin{equation*}
\#\Per_n(f)= \frac{1}{|\Gamma|}\sum_{\gamma\in\Gamma}\frac{N(\sigma^{n}-\gamma)}{p^{v_\mathfrak{p}(\sigma^{n}-\gamma)}},
\end{equation*}
and $\mathfrak{p}$ is the extension to $\mathcal{O}_K$ of the ideal in $\End(E)$ consisting of all inseparable isogenies, so $\sigma\notin\mathfrak{p}$. 

Assume for the moment that $p\notin\{2,3\}$. Let $m$ be the multiplicative order of the image of $\sigma$ in the residue field $\mathcal{O}_K/\mathfrak{p}$, so that $\sigma^m-1\in\mathfrak{p}$. By Proposition ~\ref{lem: Lifting the Exponent Number Field}, 
\begin{equation*}
v_\mathfrak{p}(\sigma^{mn}-1)=v_\mathfrak{p}(\sigma^m-1)+v_p(n).
\end{equation*}
In particular, $\sigma^{mn}-1\in\mathfrak{p}$ for $n\geq 1$. We argue that $\sigma^{mn}-\gamma\notin\mathfrak{p}$ for $\gamma\in\Gamma\setminus \{1\}$. 

If $\sigma^{mn}-\gamma$ were in $\mathfrak{p}$, then $1-\gamma$ would be in $\mathfrak{p}$, and $p=N(\mathfrak{p})$ would divide $N(1-\gamma)$. We compute
\begin{equation*}
N(1-\gamma)=(1-\gamma)(1-\hat{\gamma})=1-\tr(\gamma)+N(\gamma)=2-\tr(\gamma).
\end{equation*}
As $\gamma$ is a root of the $k$th cyclotomic polynomial for some $k\in\{2,3,4,6\}$, it follows that $\tr(\gamma)\in\{-2,-1,0,1\}$. Therefore $N(1-\gamma)\in\{1,2,3,4\}$, so $v_p(N(1-\gamma))=0$ and $1-\gamma\notin\mathfrak{p}$, so $\sigma^{mn}-\gamma\notin\mathfrak{p}$. Therefore
\begin{equation*}
\#\Per_{mn}(f)=  \frac{1}{|\Gamma|}\frac{N(\sigma^{mn}-1)}{p^{v_\mathfrak{p}(\sigma^m-1)+v_p(n)}} + \frac{1}{|\Gamma|}\sum_{\gamma\in\Gamma\setminus\{1\}}N(\sigma^{mn}-\gamma).
\end{equation*}

Let $\ell>p$ be a rational prime to be determined. For now, assume that $(\ell, N(\sigma))=(\ell,|\Gamma|)=1$. Each sequence $n\mapsto N(\sigma^{mn}-\gamma)\pmod{\ell}$ is periodic and therefore $\ell$-automatic. It follows that the sequence given by
\begin{equation*}
a_n=\frac{1}{|\Gamma|}\frac{N(\sigma^{mn}-1)}{p^{v_\mathfrak{p}(\sigma^m-1)+v_p(n)}}\pmod{\ell}
\end{equation*}
is $\ell$-automatic. By Proposition ~\ref{prop: norm recurrent}, the sequence $n\mapsto N(\sigma^{mn}-1)\pmod{\ell}$ is periodic of period dividing $(\ell-1)(\ell^2-1)\ell^A$ for some $A$. Let $c$ be this period. Choose $\ell\equiv 2\pmod{p}$ subject to the previous restrictions on $\ell$ and such that $(\ell,N(\sigma^m-1))=1$. Therefore
\begin{equation*}
N(\sigma^{m(cn+1)}-1)\equiv N(\sigma^m-1)\not\equiv 0\pmod{\ell}.
\end{equation*}
By the closure properties of Proposition ~\ref{prop: automatic closure}, the sequence given by
\begin{equation*}
n\mapsto p^{v_p(c n+1)}\pmod{\ell}
\end{equation*}
is $\ell$-automatic. We know $(p,\ell)=(p,\ell-1)=(p,\ell^2-1)=1$, so $v_p(c)=0$. So $v_p(c)\leq v_p(1)=1$, and this is a contradiction by Proposition ~\ref{prop: v_p automaticity}.

If $p\in\{2,3\}$, then $\Aut(E)\cong\mu_2$, as the only possible larger automorphism groups in characteristic 2 or 3 are nonabelian and cannot be realized as subgroups of $\mathcal{O}_K^\times$ ~\cite[Appendix A]{SilvermanEllipticCurves}. Therefore $\Gamma=\Aut(E)=\{\pm 1\}$. 

First assume that $p=3$. Everything in the above argument holds, except possibly that $c$, the period of $n\mapsto N(\sigma^{mn}-1)\pmod{\ell}$, might be divisible by $3$. Choose $\ell$ such that $\ell\equiv 2\pmod{9}$ and $(\ell,N(\sigma^{3m}-1))=1$. By Proposition ~\ref{prop: norm recurrent}, $v_3(c)\leq v_3((\ell-1)(\ell^2-1)\ell^A)=1$. The contradiction by Proposition ~\ref{prop: v_p automaticity} now comes from manipulating $(a_{cn+3})$ to produce the sequence 
\begin{equation*}
n\mapsto 3^{v_3(c n+3)}\pmod{\ell}.
\end{equation*}

Now assume that $p=2$. As $\sigma\notin\mathfrak{p}$, the image of $\sigma$ in the finite ring $\mathcal{O}_K/\mathfrak{p}^2$ is invertible. Let $m$ be the multiplicative order of the image of $\sigma$, so that $\sigma^m-1\in\mathfrak{p}^2$. Using Lemma ~\ref{lem: Lifting the Exponent Number Field} and following the above argument, we arrive at
\begin{equation*}
\#\Per_{mn}(f) =  \frac{1}{2}\left(\frac{N(\sigma^{mn}-1)}{2^{v_\mathfrak{p}(\sigma^{mn}-1)}}+\frac{N(\sigma^{mn}+1)}{2^{v_\mathfrak{p}(\sigma^{mn}+1)}}\right)
\end{equation*}
As $v_\mathfrak{p}(\sigma^m-1)\geq 2$, we have $v_\mathfrak{p}(\sigma^{mn}-1)\geq 2$ for all $n$. By properties of valuations,
\begin{equation*}
v_\mathfrak{p}(\sigma^{mn}+1)= \min(v_\mathfrak{p}(\sigma^{mn}-1),v_\mathfrak{p}(2))=1,
\end{equation*}
where $v_\mathfrak{p}(2)=1$ because $2$ splits in $K$. Therefore
\begin{equation*}
\#\Per_{mn}(f) =  \frac{1}{2}\left(\frac{N(\sigma^{mn}-1)}{2^{v_\mathfrak{p}(\sigma^{mn}-1)}}+\frac{N(\sigma^{mn}+1)}{2}\right).
\end{equation*} 

Again the only difficulty is that $c$, the period of $n\mapsto N(\sigma^{mn}-1)\pmod{\ell}$, might be even. Choose $\ell\equiv 3\pmod{8}$ such that $(\ell,N(\sigma^{16m}-1))=1$. Now $v_2(c)\leq v_2((\ell-1)(\ell^2-1)\ell^A)=4$, and the same contradiction comes from the sequence
\begin{equation*}
n\mapsto 2^{v_2(c n+16)}\pmod{\ell}.
\end{equation*}

Now suppose that $E$ is supersingular. In this case $\End(E)$ can be identified with a maximal order $\mathcal{O}$ of a rational quaternion algebra $B$ that is ramified only at $p$ and $\infty$.  Let $I$ be the two-sided maximal ideal of $\mathcal{O}$ from Proposition ~\ref{prop: inseparable degree supersingular}. As $\sigma:E\to E$ is separable, $\sigma\notin I$. Let $m$ be the multiplicative order of $\sigma$ in the residue field $\mathcal{O}/I$, so that $v_I(\sigma^m-1)\geq 1$. 

Assume first that $p\notin\{2,3\}$ so that $\Gamma\cong\mu_2$, $\mu_3$, $\mu_4$, or $\mu_6$. By Proposition ~\ref{prop: inseparable degree supersingular} and Lemma ~\ref{lem: Lifting the Exponent Quaternion Algebra},
\begin{equation*}
v_I(\sigma^{mn}-1)=v_I(\sigma^m-1) + 2v_p(n).
\end{equation*}
As in the case of $E$ ordinary, for $\xi\in\Gamma\setminus\{1\}$, $\xi$ is a root of the $k$th cyclotomic polynomial some for $k\in\{2,3,4,6\}$, so $N(1-\xi)\in\{1,2,3,4\}$. Norm considerations show that $\sigma^{mn}-\xi\notin I$ for $n\geq 1$, so $\#\ker(\sigma^{mn}-\xi)=N(\sigma^{mn}-\xi)$. Therefore
\begin{equation*}
\#\Per_{mn}(f) =  \frac{1}{|\Gamma|}\frac{N(\sigma^{mn}-1)}{p^{v_I(\sigma^m-1)+2v_p(n)}} + \frac{1}{|\Gamma|}\sum_{\xi\in\Gamma\setminus\{1\}}N(\sigma^{mn}-\xi).
\end{equation*}
The same reasoning as in the ordinary case shows that there is a prime $\ell$ such that $\#\Per_{mn}(f)\pmod{\ell}$ is not $\ell$-automatic. 

Now assume that $p\in\{2,3\}$. If $j(E)\neq 0$, then $\Aut(E)=\{\pm 1\}$, and the argument proceeds exactly as when $E$ is ordinary. Therefore assume that $j(E)=0$. For both $p=2$ and $p=3$, the maximal order $\mathcal{O}$ has trivial class group and so is the unique maximal order of $B$ up to conjugacy ~\cite[Ch 1, Cor 4.11]{Vigneras}. Therefore, for the purposes of identifying $\End(E)$ with $\mathcal{O}$, we may take $\mathcal{O}$ to be any maximal order of $B$.

First let $p=2$. In this case, $B$ is the Hamilton quaternions $\left(\frac{-1,-1}{\Q}\right)=\Q(i,j)$ where $i^2=j^2=-1$ and $k=ij$. A maximal order of $B$ is given by the Hurwitz quaternions
\begin{equation*}
\mathcal{O} = \Z i+ \Z j + \Z k + \Z(1+i+j+k)/2.
\end{equation*}
Here $\Aut(E)\cong\SL_2(\F_3)$ can be described explicitly as
\begin{equation*}
\Aut(E)\cong\mathcal{O}^\times=\{\pm 1, \pm i, \pm j, \pm k, (\pm 1\pm i\pm j\pm k)/2\}.
\end{equation*}
For $\gamma\in\Aut(E)$, explicit calculation of norms shows that $v_I(1-\gamma)=0$ unless $\gamma\in\{\pm 1,\pm i,\pm j, \pm k\}\cong Q_8$, which is the unique Sylow 2-subgroup of $\Aut(E)$. As $\sigma\notin I$, the image of $\sigma$ is invertible in $\mathcal{O}/I^3$. Pick $m$ such that $\sigma^m- 1\in I^3$, i.e. $v_I(\sigma^m-1)\geq 3$. Then for $\gamma\neq 1$,
\begin{equation*}
   v_I(\sigma^{mn}-\gamma)=\min\{v_I(\sigma^{mn}-1),v_I(1-\gamma)\} = \left\{
     \begin{array}{ll}
       2 & : \gamma =-1 \\
       1 & : \gamma= \pm i, \pm j, \pm k\\
       0 & : \text{otherwise}
     \end{array}
   \right.
\end{equation*}
So for each $\gamma\in\Gamma$, there exists a constant $C(\gamma)$ such that
\begin{equation*}
\#\Per_{mn}(f) =  \frac{1}{|\Gamma|}\left(\frac{N(\sigma^{mn}-1)}{2^{v_I(\sigma^m-1)+2v_2(n)}} + \sum_{\gamma\in\Gamma\setminus\{1\}}\frac{N(\sigma^{nm}-\gamma)}{C(\gamma)}\right)
\end{equation*}
Choosing $\ell$ appropriately, we can reduce this to the sequence
\begin{equation*}
n\mapsto 4^{v_2(cn + 16)}
\end{equation*}
and get a contradiction as before.

Now let $p=3$, so that $B=\left(\frac{-1,-3}{\Q}\right)=\Q(i,j)$ where $i^2=-1$, $j^2=-3$, and $k=ij$. A maximal order $\mathcal{O}$ is given by 
\begin{equation*}
\mathcal{O} = \Z + \Z i + \Z(1+j)/2 + \Z(i+k)/2
\end{equation*}
and again $\Aut(E)\cong C_3\rtimes C_4$ has an explicit description as
\begin{equation*}
\Aut(E)\cong\mathcal{O}^\times = \{\pm 1, \pm i, (\pm 1\pm j)/2, (\pm i\pm k)/2 \}.
\end{equation*}
For $\gamma\in\Aut(E)$, another norm calculation shows that $v_I(1-\gamma)=0$ unless $\gamma\in\{1,\frac{-1\pm j}{2}\}\cong C_3$, the unique Sylow 3-subgroup of $\Aut(E)$. Pick $m$ such that $v_I(\sigma^m-1)\geq 2$. If $\gamma\neq 1$,
\begin{equation*}
   v_I(\sigma^{mn}-\gamma)=\min\{v_I(\sigma^{mn}-1),v_I(1-\gamma)\} = \left\{
     \begin{array}{ll}
       1 & : \gamma =(-1\pm j)/2 \\
       0 & : \text{otherwise}
     \end{array}
   \right.
\end{equation*}
As above, there are $C(\gamma)$ so that
\begin{align*}
\#\Per_{mn}(f) =  \frac{1}{|\Gamma|}&\left(  \frac{N(\sigma^{mn}-1)}{3^{v_I(\sigma^m-1)+2v_3(n)}} + \sum_{\gamma\in\Gamma\setminus\{1\}} \frac{N(\sigma^{mn}-\gamma)}{C(\gamma)}  )\right).
\end{align*}
We can reduce this to the sequence
\begin{equation*}
n\mapsto 9^{v_3(c_n+3)}
\end{equation*}
and the same argument gives a contradiction.

\section{Proof of Theorem 1.3}\label{sec: Proof 1.3}

Let $k$ be an algebraically closed field of characteristic $p$, and let $f\in k[x]$ be an additive or subadditive polynomial. The map $f$ fits into Diagram ~\ref{eqn: dynamically affine} with $G=\G_a$, $\Gamma\cong\mu_d$, and $\pi(x)=x^d$, where possibly $d=1$. Therefore $\psi(x)^d=f(x^d)$. As usual, let $\psi$ be the endomorphism $\sigma$ composed with a translation. 

Let $a$ be the constant term of $\sigma\in k\langle\phi\rangle$, so that $\sigma=a+(\phi)$ (that is, $\sigma(x)=ax + g(x)$ for some $g\in(\phi)$). If $d=1$, then $f(x)$ is $\sigma(x)$ composed with a translation, so $f'(0)=a$. If $d\geq 2$, then $\psi(\omega_d x)=\omega_d\psi(x)$, so $\psi(0)=0$ and $\psi=\sigma$. So $\sigma^d=a^d + (\phi)$, and $f'(0)=a^d$. Therefore $f'(0)$ is algebraic if and only if $a$ is algebraic.

By equation ~\ref{eqn: periodic count additive polynomial}, 
\begin{equation*}
\#\Per_n(f) = 1 + \frac{1}{d}\sum_{\omega\in\mu_d}\frac{(\deg\sigma)^n}{p^{v_\phi(\sigma^n-\omega)}}.
\end{equation*}
As $\sigma$ is separable, $v_\phi(\sigma)=0$, so $a\neq 0$.

If $f'(0)$ is transcendental, then so is $a$. The constant term of $\sigma^n-\omega$ is $a^n-\omega$, which is never zero (if it were, $a$ would be algebraic). Therefore
\begin{equation*}
\#\Per_n(f) = 1 + \frac{1}{d}\sum_{\omega\in\mu_d}(\deg\sigma)^n = 1+(\deg\sigma)^n.
\end{equation*}
It follows easily that $\zeta(f,\P^1(k);t)$ is rational.

Suppose $f'(0)$ is algebraic over $\F_p$, so $a$ is algebraic and therefore is a root of unity. Aiming for a contradiction, assume that $\zeta(f,\P^1(k);t)$ is algebraic. There exists $m$ such that the image of $\sigma^m$ in $k\langle\phi\rangle/(\phi)$ is 1, and so $\sigma^m-1\in(\phi)$. By Lemma ~\ref{lem: Lifting the Exponent G_a}, for $n\geq 1$,
\begin{equation*}
v_\phi(\sigma^{mn}-1)=v_\phi(\sigma^m-1)p^{v_p(n)}.
\end{equation*}
In particular, $\sigma^{mn}-1\in(\phi)$. Therefore, for $\omega\in\mu_d\setminus\{1\}$, we have $\sigma^{mn}-\omega\notin (\phi)$, because otherwise $1-\omega$ would be in $(\phi)$. But $1-\omega\in k\langle\phi\rangle$ represents the linear polynomial $x\mapsto (1-\omega)x$, which does not factor through $\phi:x\mapsto x^p$. Therefore
\begin{equation*}
\#\Per_{mn}(f) = 1 + \frac{1}{d}\left(\frac{(\deg\sigma)^n}{p^{v_\phi(\sigma^m-1)p^{v_p(n)}}} + \sum_{\omega\in\mu_d\setminus\{1\}}(\deg\sigma)^n\right).
\end{equation*}
Let $\ell$ be a large prime. If $p$ is odd, choose $\ell$ such that $\ell\equiv 2\pmod{p}$, and if $p=2$, choose $\ell\equiv 7\pmod{8}$. Let $a_n=\#\Per_{mn}\pmod{\ell}$. Note that $n\mapsto (\deg\sigma)^n$ is periodic and therefore $\ell$-automatic. By Proposition ~\ref{prop: automatic closure}, we can manipulate $a_n$ to arrive at the sequence
\begin{equation*}
b_n=(\deg\sigma)^n p^{-p^{\left(v_\phi(\sigma^m-1)p^{v_p(n)}\right)}},
\end{equation*}
which is $\ell$-automatic. The subsequence $b_{(\ell-1)n}$ is $\ell$-automatic, as is its reciprocal
\begin{equation*}
(b_{(\ell-1)n})^{-1} = p^{p^{\left(v_\phi(\sigma^m-1)p^{v_p(n)}\right)}}.
\end{equation*}
For a large enough $\ell$, the above sequence satisfies the assumptions of Proposition ~\ref{prop: v_p automaticity G_a}, so it is not $\ell$-automatic, which is a contradiction.

\subsection*{Acknowledgements}
We would like to thank Michael Zieve for drawing our attention to the results in ~\cite{GhiocaZieve}, Tonghai Yang for helpful advice regarding quaternion algebras,  and Bjorn Poonen for pointing out the counterexample to Conjecture ~\ref{conj: separable} in the case where $k$ contains transcendentals over $\F_p$. This research was partly supported by NSF Grant no. EMSW21-RTG and by the Wisconsin Alumni Research Foundation.

\bibliographystyle{plain}
\bibliography{new_bib}
\nocite{*}

\end{document}